\documentclass{file}
\usepackage{pat}
\usepackage{paralist}
\usepackage{dsfont}  
\usepackage[utf8x]{inputenc}
\usepackage{skull}
\usepackage{paralist}
\usepackage{centernot}
\usepackage{mathtools}
\usepackage{thmtools}
\usepackage{hyperref}
\usepackage[table]{xcolor}
\usepackage{stmaryrd,tikz}
\usetikzlibrary{calc}
\usepackage{graphicx}
\usepackage{booktabs}
\usepackage{listings}
\usepackage{pifont}%
\usepackage[normalem]{ulem}
\usepackage{cancel,float}
\usepackage{comment}

\usepackage{todonotes}
\usepackage{mathdots}
\usepackage{bm}
\usepackage{etoolbox}
\definecolor{maroon}{rgb}{0.5, 0.0, 0.0}
\definecolor{darkblue}{rgb}{0.0, 0.0, 0.55}

\usepackage[longnamesfirst,numbers,sort&compress]{natbib}

\usepackage[mathlines]{lineno}
\newcommand*\patchAmsMathEnvironmentForLineno[1]{%
 \expandafter\let\csname old#1\expandafter\endcsname\csname #1\endcsname
 \expandafter\let\csname oldend#1\expandafter\endcsname\csname end#1\endcsname
 \renewenvironment{#1}%
    {\linenomath\csname old#1\endcsname}%
    {\csname oldend#1\endcsname\endlinenomath}}%
\newcommand*\patchBothAmsMathEnvironmentsForLineno[1]{%
 \patchAmsMathEnvironmentForLineno{#1}%
 \patchAmsMathEnvironmentForLineno{#1*}}%
\AtBeginDocument{%
\patchBothAmsMathEnvironmentsForLineno{equation}%
\patchBothAmsMathEnvironmentsForLineno{align}%
\patchBothAmsMathEnvironmentsForLineno{flalign}%
\patchBothAmsMathEnvironmentsForLineno{alignat}%
\patchBothAmsMathEnvironmentsForLineno{gather}%
\patchBothAmsMathEnvironmentsForLineno{multline}%
}

\definecolor{brightmaroon}{rgb}{0.76, 0.13, 0.28}
\definecolor{linkblue}{rgb}{0, 0.337, 0.227}
\newcommand{\eps}{\varepsilon}

\newcommand{\supp}{\operatorname{supp}}

\newcommand{\Cay}{\operatorname{Cay}}
\newcommand{\Good}{\operatorname{Good}}

\newrobustcmd{\onesub}{\mathord{\includegraphics{one-sub}}}
\newrobustcmd{\leftup}{\mathord{\includegraphics{left-up}}}
\makeatletter
\newcommand{\xMapsto}[2][]{\ext@arrow 0599{\Mapstofill@}{#1}{#2}}
\def\Mapstofill@{\arrowfill@{\Mapstochar\Relbar}\Relbar\Rightarrow}
\makeatother

\title{\MakeUppercase{On {\Large{$\varepsilon$}}-Matrix Product Factorization of graphs}}
\author{Farzad Maghsoudi\,\,\,\,\,Bobby Miraftab\thanks{School of Computer Science, Carleton University, Ottawa, ON, Canada.}  \,\,\,\,\, Sho Suda\thanks{Department of Mathematics, National Defense Academy of Japan, Yokosuka, Kanagawa 239-8686,
Japan}}

\date{}

\begin{document}

\maketitle

\begin{abstract}
We introduce an approximate version of matrix product factorization for graphs. A simple graph $G$ on $n$ vertices is said to admit an $\varepsilon$-matrix product factorization if there exist simple graphs $H$ and $K$ on the same vertex set such that $A(H)A(K)$ and $A(G)$ disagree in at most $\varepsilon n^{2}$ entries. This Hamming-type relaxation preserves, outside the error set, the exact interpretation of each edge as having a unique $H$-then-$K$ two-step witness. We establish equivalent matrix, and witness formulations, showing that the sets $N_H(w)\times N_K(w)$ form an approximate disjoint decomposition of the ordered adjacency relation of $G$, and we derive quantitative constraints involving walk counts and the degrees of the factor graphs.

We then construct approximate factorizations for several graph families. Every complete graph $K_n$ has matrix-product-factorization distance $O(1/n)$, despite the exact congruence obstruction that permits exact factorization only when $n\equiv 1\pmod 4$. More generally, a blow-up of a fixed graph on $r$ vertices admits an $\varepsilon$-factorization with $\varepsilon\le r/n$, and the construction is exact whenever every non-isolated part has even order. For bipartite graphs, we give one-sided factorizations that realize one orientation of almost all edges. In particular, every tree on $n\ge2$ vertices admits an $\varepsilon$-factorization with $\varepsilon\le 1/n$, although no nontrivial tree is exactly factorizable. These results show that rigid exact obstructions may disappear under a vanishing proportion of entrywise errors.
\end{abstract}
\section{Introduction}
\label{sec:introduction}

Let $G$ be a simple graph on an ordered vertex set
$V(G)=\{v_1,\dots,v_n\}$.  If $H$ and $K$ are simple graphs on the same
vertex set, then
\[
  \bigl(A(H)A(K)\bigr)_{ij}
  =
  \sum_{\ell=1}^n A(H)_{i\ell}A(K)_{\ell j}
  =
  \bigl|N_H(v_i)\cap N_K(v_j)\bigr|.
\]
Thus, the $(i,j)$-entry counts the two-step walks from $v_i$ to $v_j$ that first use an edge of $H$ and then an edge of $K$. Although $A(H)$ and $A(K)$ are symmetric 0-1 matrices with zero diagonal, their product need not have these properties. Therefore, the condition $A(H)A(K)=A(G)$ is strong: each ordered edge of $G$ must have exactly one such two-step walk, while each ordered non-edge, including each diagonal pair, must have none.
Moreover, because $A(G)$ is symmetric, an exact factorization forces $A(H)A(K)=A(K)A(H)$.

The study of products of adjacency matrices has its roots in work on
commuting graphs and commuting decompositions
\cite{akbari2007commuting,akbari2009commutativity,somodi2017construction}.
The matrix product of labelled graphs was introduced by Prasad, Sudhakara,
Sujatha, and Vinay \cite{prasad2013matrix}, who characterized when the product of two adjacency matrices is graphical and related this condition to an alternating two-colour, or diamond condition.  A
modulo-2 version was subsequently investigated in
\cite{prasad2014modulo}. Related matrix equations, products involving $k$-complements, commuting decompositions of complete multipartite graphs,
and the companion-existence problem were studied in
\cite{bhat2016matrix,bhat2018commuting,bhat2022algorithm}; this first phase
of the subject is surveyed in \cite{sudhakara2023products}.

Determining which graphs can be written as a product of two adjacency matrices was developed by Maghsoudi, Miraftab, and Suda \cite{maghsoudi2023matrix}. They introduced
the class of matrix-product-factorizable graphs, classified complete graphs
and complete bipartite graphs, and obtained constructions for several further graph families. In particular, $K_n$ is exactly factorizable if and only if $n\equiv 1\pmod 4$, whereas $K_{m,n}$ is exactly
factorizable if and only if both parts have even size. Akbari, Fan, Hu, Miraftab, and Wang developed spectral methods for the exact problem and,
among other results, proved that every nontrivial tree is not exactly factorizable \cite{miraftab2025factorability}. The theory has also been extended to infinite graphs \cite{miraftab2026infinite}.  More recent work
characterizes prime factorizations and factorable forests, tori, and grids
\cite{akbari2025prime}, formulates factorability of Cayley graphs through unique group factorizations and Sidon-pair conditions
\cite{herman2025cayley}, and studies matrix product factorizations inside symmetric association schemes \cite{herman2026association}.

All of these results deal with exact equality. However, exact factorization does not show how far a graph is from being factorizable. A graph may fail the product condition on many of its $n^2$ ordered pairs, or only on a small number because of a parity, arithmetic, or boundary issue. Both cases are simply called “not factorizable.” The
complete-graph classification already suggests the issue: the congruence
condition $n\equiv1\pmod4$ is an absolute obstruction to exact
factorization, but deleting a bounded number of vertices from a nearby exact
construction removes only $O(n)$ of the $n^2$ matrix entries. Similar behavior occurs for blow-ups with odd parts and for sparse bipartite graphs
This motivates a stability version of matrix-product factorization, where the product condition is required to hold on almost all ordered pairs.

\begin{defn}
\label{def:eps-factorization}
Let $0\leq \eps\leq 1$, and let $G$ be a simple graph on $n$ vertices.  We
say that $G$ admits an \defin{$\eps$-matrix product factorization}, or simply
an \defin{$\eps$-factorization}, if there exist simple graphs $H$ and $K$ on
$V(G)$ and a set of ordered pairs
$\mathcal E\subseteq V(G)\times V(G)$, and $|\mathcal E|\leq \eps n^2$,
such that, for every $(u,v)\notin\mathcal E$,
\[
  |N_H(u)\cap N_K(v)|
  =
  \mathbf 1_{\{uv\in E(G)\}}.
\]
The set $\mathcal E$ is called an \defin{error set}, and $H,K$ are called
\defin{factor graphs}.
\end{defn}

Equivalently, after choosing a common ordering of the vertices, the matrices
$A(H)A(K)$ and $A(G)$ disagree in at most $\eps n^2$ positions.  It is useful
to regard
\[
  \delta_{\mathrm{MPF}}(G)
  :=
  \frac{1}{n^2}
  \min_{H,K}
  \bigl|
    \{(i,j)\in[n]^2:
       (A(H)A(K))_{ij}\neq A(G)_{ij}\}
  \bigr|
\]
as the normalized matrix-product-factorization distance of $G$; then $G$
admits an $\eps$-factorization precisely when
$\delta_{\mathrm{MPF}}(G)\leq\eps$.

The particular error model in Definition~\ref{def:eps-factorization} is
deliberate. We measure the support of the discrepancy rather than a
Frobenius or operator norm because the defining information is local and
combinatorial.  At every non-error pair, an edge still has exactly one
witness and a non-edge still has none; an entry equal to $2$, for example,
is recorded as a failure of uniqueness rather than as a small numerical
perturbation of an entry equal to $1$.  The errors are ordered because
$A(H)A(K)$ need not be symmetric, even though $H$ and $K$ are undirected.
Finally, the normalization by $n^2$ is the natural dense-matrix
normalization and makes the notion compatible with Hamming-type stability
and graph-edit questions.  The case $\eps=0$ recovers the exact theory.

Now, if $G$ has $m$ edges, taking one
factor graph to be empty gives the zero product and hence the elementary
bound
\[
  \delta_{\mathrm{MPF}}(G)\leq \frac{2m}{n^2}.
\]
Consequently, the normalized parameter alone is automatically small for
very sparse graphs.  For this reason, throughout the paper we retain the
unnormalized number of error entries and seek structured factorizations that
preserve a substantial part of the adjacency relation, rather than relying
only on the zero-product bound.  This distinction is particularly relevant
for trees and sparse bipartite graphs.

We now summarize our main results. First, we establish several structural
forms of approximate factorization.  Outside the error set, the product has
the same unique-witness interpretation as in the exact theory.  For each
vertex $w$, $N_H(w)\times N_K(w)$ consists of the ordered pairs witnessed by $w$. We also derive walk-count and
row-support estimates, including quantitative constraints involving the
degrees and maximum degrees of the factor graphs.

Second, we show that the exact congruence obstruction for complete graphs is
unstable in matrix Hamming distance.  More precisely, every complete graph
admits an $\eps$-factorization with
\[
\eps\leq
\begin{cases}
0, & n\equiv1\pmod4,\\[1mm]
\dfrac1n, & n\equiv0\pmod4,\\[2mm]
\dfrac1n, & n\equiv2\pmod4,\\[2mm]
\dfrac{n+1}{n^2}, & n\equiv3\pmod4.
\end{cases}
\]
Thus every $K_n$ is at normalized distance $O(1/n)$ from a product of two
adjacency matrices, even though three of the four congruence classes are
excluded from exact factorization.

Third, we develop a blow-up construction.  If $G$ is a blow-up of a fixed
graph $F$ on $r$ vertices, with parts $V_1,\dots,V_r$, then
\[
  \eps
  \leq
  \frac{1}{n^2}
  \sum_{\substack{b\in[r]\\ |V_b|\,\mathrm{odd}}}
  \sum_{a\in N_F(b)} |V_a|
  \leq \frac{r}{n}.
\]
The construction is exact whenever every non-isolated part has even size.
Complete multipartite graphs arise as an immediate special case.  This
result identifies the unmatched vertices in odd parts as the only source of
error and gives a general stability mechanism for template-based graph
families.

Finally, for a bipartite graph with $m$ edges we construct factors that
realize one ordered orientation of almost all edges exactly.  
When one bipartition class has even size, the construction uses exactly $m$ error
entries; after removing one vertex $y_0$ from an odd class, it uses
$m+d_G(y_0)$ errors.  
Applied to trees, this yields $\eps\leq \frac1n$, with the sharper bound $(n-1)/n^2$ whenever one bipartition class has even cardinality.  
Hence the rigid exact obstruction for trees coexists with a uniformly vanishing approximate error, while the explicit error count shows that the construction is genuinely stronger than simply discarding the whole adjacency matrix.

The paper is organized as follows.  In~\Cref{sec:basic} we give the
matrix, and witness formulations of $\eps$-factorization and derive basic quantitative constraints.  
~\Cref{sec:complete} treats
complete graphs.  
~\Cref{sec:blowups} develops the blow-up and complete multipartite constructions.  ~\Cref{sec:trees} considers bipartite
graphs and trees.

\section{Basic Structural Properties}
\label{sec:basic}

We start with some basic definitions and notations.
For two matrices $B=(b_{ij})$ and $C=(c_{ij})$ of the same size,
their \defin{Hadamard product} is the entrywise product
\[
  B\odot C := (b_{ij}c_{ij}).
\]
For a matrix $M=(m_{ij})$, we write $\|M\|_0:=|\{(i,j):m_{ij}\neq 0\}|$ for the number of non-zero entries of $M$.  In particular, if
$E\in\{0,1\}^{n\times n}$, then
\[
  \|E\|_0=\sum_{i,j}E_{ij}.
\]
We write $J=J_n$ for the $n\times n$ all-ones matrix.

\begin{lem}
\label{lem:hadamard-form}
Let $G$ be a simple graph on $n$ vertices.  Then $G$ has an
$\eps$-factorization if and only if, after choosing an ordering
$V(G)=\{v_1,\dots,v_n\}$, there exist simple graphs $H,K$ on $V(G)$ and a
matrix $E\in\{0,1\}^{n\times n}$ such that
\begin{enumerate}
\item $\|E\|_0\le \eps n^2$.
\item $ \bigl(A(H)A(K)-A(G)\bigr)\odot (J-E)=0$.
\end{enumerate}
\end{lem}

\begin{proof}
Fix an ordering $V(G)=\{v_1,\dots,v_n\}$.  For every $i,j$,
\[
  (A(H)A(K))_{ij}
  =
  \sum_{\ell=1}^n A(H)_{i\ell}A(K)_{\ell j}
  =
  |N_H(v_i)\cap N_K(v_j)|.
\]
Also, we have  $A(G)_{ij}=\mathbf 1_{\{v_iv_j\in E(G)\}}$.
The Hadamard identity is equivalent to saying that the two matrices agree at all entries where $E$ is zero.  Its $(i,j)$-entry is $\bigl((A(H)A(K))_{ij}-A(G)_{ij}\bigr)(1-E_{ij})$.
Thus, if $E_{ij}=0$, then $(A(H)A(K))_{ij}=A(G)_{ij}$,
while entries with $E_{ij}=1$ are ignored.
If the matrix condition holds, define $\mathcal E:=\{(v_i,v_j):E_{ij}=1\}$.
Then $|\mathcal E|=\|E\|_0\le \eps n^2$.
For every $(v_i,v_j)\notin\mathcal E$, we have $E_{ij}=0$, and hence we obtain
\[
  |N_H(v_i)\cap N_K(v_j)|
  =
  (A(H)A(K))_{ij}
  =
  A(G)_{ij}
  =
  \mathbf 1_{\{v_iv_j\in E(G)\}}.
\]
Therefore $H,K,\mathcal E$ give an $\eps$-factorization.

Conversely, let $H,K$ be the factor graphs and let
$\mathcal E\subseteq V(G)\times V(G)$
be the error set.  Define $E\in\{0,1\}^{n\times n}$ by
\[
  E_{ij}=1
  \quad\Longleftrightarrow\quad
  (v_i,v_j)\in\mathcal E.
\]
Then we have  $\|E\|_0=|\mathcal E|\le \eps n^2$.
Moreover, if $E_{ij}=0$, then $(v_i,v_j)\notin\mathcal E$, so the definition gives
\[
  |N_H(v_i)\cap N_K(v_j)|
  =
  \mathbf 1_{\{v_iv_j\in E(G)\}}.
\]
Equivalently, we have  $(A(H)A(K))_{ij}=A(G)_{ij}$ whenever $E_{ij}=0$.
Thus we obtain
\[
  \bigl(A(H)A(K)-A(G)\bigr)\odot (J-E)=0.\qedhere
\]
\end{proof}

\begin{defn}
For a fixed matrix-form $\eps$-factorization $(H,K,E)$, let
\[
  \Good:=\{(i,j)\in[n]\times[n]:E_{ij}=0\}
\]
be the set of non-error entries.
\end{defn}

In the next lemma, we show that for every $(i,j)\in\Good$, there is at most one vertex
$v_\ell$ such that
\[
  v_iv_\ell\in E(H)
  \quad\text{and}\quad
  v_\ell v_j\in E(K).
\]
\begin{lem}
\label{lem:unique-witness}
Suppose $G$ has an $\eps$-factorization $(H,K,E)$ with respect to the ordering $V(G)=\{v_1,\dots,v_n\}$.  Then for
every $(i,j)\in\Good$,
\[
  |N_H(v_i)\cap N_K(v_j)|=A(G)_{ij}\in\{0,1\}.
\]

\end{lem}

\begin{proof}
If $(i,j)\in\Good$, then $E_{ij}=0$.  
It follows from \Cref{lem:hadamard-form}, that
$(A(H)A(K))_{ij}=A(G)_{ij}$.
Since $G$ is simple, $A(G)_{ij}\in\{0,1\}$.  
On the other hand, we have
\[
  (A(H)A(K))_{ij}
  =
  \sum_{\ell=1}^n A(H)_{i\ell}A(K)_{\ell j}
  =
  |\{v_\ell:v_iv_\ell\in E(H),\ v_\ell v_j\in E(K)\}|.
\]
Therefore the number of such middle vertices is either $0$ or $1$.
\end{proof}

The previous lemma is the local witness form of the exact theory.  In the
two-colouring language used in the diamond-condition formulation, one colours
the edges coming from one factor blue and the edges coming from the other
factor red.  Then an entry of $A(H)A(K)$ counts blue-red alternating two-step
walks.  Exact factorization requires these alternating witnesses to occur
exactly where the target adjacency matrix has a $1$.  The next observation is the approximate version of the same statement.

Before relating the product to the colouring or diamond-condition viewpoint, we
fix a convention about ordered matrix entries.  All graphs in this paper are
simple and undirected.  Thus, if $v_iv_j\in E(G)$, then both adjacency-matrix
entries
\[
  A(G)_{ij}=1
  \qquad\text{and}\qquad
  A(G)_{ji}=1
\]
are present.  
No orientation of $G$ is being introduced.
We write
\[
  \mathcal A(G)\coloneqq \{(i,j)\in[n]\times[n]: A(G)_{ij}=1\}
  =
  \{(i,j): i\neq j,\ v_iv_j\in E(G)\}
\]
for the set of ordered adjacency entries of $G$.  Thus every undirected edge
$v_iv_j$ contributes the two ordered entries $(i,j)$ and $(j,i)$ to
$\mathcal A(G)$.  We also write
\[
  \mathcal Z(G):=\{(i,j)\in[n]\times[n]: A(G)_{ij}=0\}
\]
for the set of zero entries of the adjacency matrix.  
This set includes the diagonal entries.

The reason ordered entries matter is that the product $A(H)A(K)$ need not be
symmetric, even though $A(H)$ and $A(K)$ are symmetric.  
Therefore equality with $A(G)$ must be checked separately at the two entries $(i,j)$ and $(j,i)$.
Similarly, when we write
\[
  v_i \xrightarrow{H} v_\ell \xrightarrow{K} v_j,
\]
the arrows do not mean that the edges of $H$ or $K$ are oriented.  They only
record the order of multiplication in $A(H)A(K)$: first one uses an edge of
$H$, and then one uses an edge of $K$.  If an edge belongs to both $H$ and $K$,
then it may be used in either role.

\begin{lem}
\label{lem:approx-diamond}
Let $(H,K,E)$ be an $\eps$-factorization of $G$.
Then the following holds:
\begin{enumerate}
  \item if $(i,j)\in \Good\cap\mathcal A(G)$, then there is exactly one
  $H$-then-$K$ path from $v_i$ to $v_j$;
  \item if $(i,j)\in \Good\cap\mathcal Z(G)$, then there is no $H$-then-$K$ path from $v_i$ to $v_j$.
\end{enumerate}
Thus the entrywise witness condition, which is the matrix-product version of
the diamond condition in the exact theory, is allowed to fail only on the
entries marked by $E$.
\end{lem}

\begin{proof}
For every ordered pair $(i,j)$, the number of $H$-then-$K$ witnesses from
$v_i$ to $v_j$ is
\[
  |\{v_\ell: v_iv_\ell\in E(H),\ v_\ell v_j\in E(K)\}|
  =
  \sum_{\ell=1}^n A(H)_{i\ell}A(K)_{\ell j}
  =
  (A(H)A(K))_{ij}.
\]
Since $(H,K,E)$ is an $\eps$-factorization, we have
\[
  (A(H)A(K))_{ij}=A(G)_{ij}
\]
for every $(i,j)\in\Good$.  If $(i,j)\in\mathcal A(G)$, then
$A(G)_{ij}=1$, so there is exactly one witness.  If
$(i,j)\in\mathcal Z(G)$, then $A(G)_{ij}=0$, so there is no witness.  Hence
all possible failures are confined to the marked entries.
\end{proof}

\begin{thm}
\label{thm:biclique-localized}
Suppose $G$ admits an $\eps$-factorization via $(H,K,E)$.  For each $w\in V(G)$, define the witness set
\[
  S_w:=N_H(w)\times N_K(w)\subseteq V(G)\times V(G).
\]
For an ordered pair $(u,v)\in V(G)\times V(G)$, define
$m(u,v):= |\{w\in V(G):(u,v)\in S_w\}|$.
Then the following hold:
\begin{enumerate}
  \item For every non-error entry $(u,v)\in\Good$, $m(u,v)=A(G)_{uv}$.
  \item If $U:=\bigcup_{w\in V(G)} S_w$,
  then $|U\triangle \mathcal A(G)|\le \eps n^2$.

  \item $\{(u,v)\in V(G)\times V(G):m(u,v)\ge 2\}\subseteq \supp(E)$.
  \item If $\eps=0$, then $\mathcal A(G)=\bigsqcup_{w\in V(G)} S_w$.
\end{enumerate}
\end{thm}

\begin{proof}
For every ordered pair $(u,v)\in V(G)\times V(G)$, we have the following:
\[
  (A(H)A(K))_{uv}
  =
  \sum_{w\in V(G)} A(H)_{uw}A(K)_{wv}.
\]
Since $H$ and $K$ are undirected,
we have $A(H)_{uw}=1$ if and only if $u\in N_H(w)$,
and moreover  $A(K)_{wv}=1$ if and only if $v\in N_K(w)$.
Therefore we obtain
\[
  (A(H)A(K))_{uv}
  =
  |\{w\in V(G):u\in N_H(w),\ v\in N_K(w)\}|
  =
  m(u,v).
\]

Because $(H,K,E)$ is an $\eps$-factorization, we have
$(A(H)A(K))_{uv}=A(G)_{uv}$ for every non-error entry $(u,v)\in\Good$.  
Hence $m(u,v)=A(G)_{uv}$ for every $(u,v)\in\Good$.  
This proves item~(1).

Now let $U=\bigcup_{w\in V(G)} S_w$.
It follows from the definition that an ordered pair $(u,v)$ belongs to $U$ if and only if $m(u,v)\ge 1$.  
On the other hand, outside the error set, we have
$m(u,v)=A(G)_{uv}\in\{0,1\}$.
Thus, for every $(u,v)\notin\supp(E)$,
\[
  (u,v)\in U
  \quad\Longleftrightarrow\quad
  m(u,v)=1
  \quad\Longleftrightarrow\quad
  A(G)_{uv}=1
  \quad\Longleftrightarrow\quad
  (u,v)\in\mathcal A(G).
\]
Therefore we obtain $U\triangle \mathcal A(G)\subseteq \supp(E)$.
Since $\|E\|_0\le \eps n^2$, it follows that[
$|U\triangle \mathcal A(G)|
  \le
  |\supp(E)|
  =
  \|E\|_0
  \le
  \eps n^2$.
This proves item~(2).

Next, if $(u,v)\notin\supp(E)$, then $(u,v)\in\Good$, and so $m(u,v)=A(G)_{uv}\in\{0,1\}$.
Hence no non-error entry can lie in two or more of the witness sets.  
Therefore we have 
\[
  \{(u,v)\in V(G)\times V(G):m(u,v)\ge 2\}\subseteq \supp(E),
\]
which proves item~(3).

Finally, suppose $\eps=0$.  Then $\|E\|_0=0$, so $\supp(E)=\emptyset$.  By
item~(2), we have $U=\mathcal A(G)$, and by item~(3), no ordered pair lies in more than one witness set.  
Hence the union is disjoint, and we obtain
\[
  \mathcal A(G)
  =
  \bigsqcup_{w\in V(G)} S_w
  =
  \bigsqcup_{w\in V(G)} N_H(w)\times N_K(w).
\]
This proves item~(4).
\end{proof}

The next estimates provide two elementary constraints on possible factors.  They
are useful both as obstructions and as diagnostics for candidate
factorizations.

\begin{lem}
\label{lem:walk-count}
\label{lem:degree-product-error}
Suppose $A(G)$ admits an $\eps$-factorization via $(H,K,E)$, and put $\Delta_*:=\min\{\Delta(H),\Delta(K)\}$.
Then
\[
  \left|
  2e(G)-\sum_{w=1}^n d_H(w)d_K(w)
  \right|
  \le
  (1+\Delta_*)\|E\|_0
  \le
  (1+\Delta_*)\eps n^2.
\]
In particular, if $\eps=0$, then $2e(G)=\sum_{w=1}^n d_H(w)d_K(w)$.
\end{lem}

\begin{proof}
First, expanding the matrix product and exchanging the order of summation gives
\[
\begin{aligned}
\sum_{i,j}(A(H)A(K))_{ij}
&=
\sum_{i,j}\sum_{w=1}^n A(H)_{iw}A(K)_{wj}  \\
&=
\sum_{w=1}^n
\left(\sum_i A(H)_{iw}\right)
\left(\sum_j A(K)_{wj}\right)  \\
&=
\sum_{w=1}^n d_H(w)d_K(w).
\end{aligned}
\]
Now let $B:=A(H)A(K)$.
Since $(H,K,E)$ is an $\eps$-factorization of $G$, we have
$B_{ij}=A(G)_{ij}$ for every non-error entry, that is, for every $(i,j)$ with $E_{ij}=0$.
Therefore
\[
  \left|
  \sum_{i,j}A(G)_{ij}-\sum_{i,j}B_{ij}
  \right|
  \le
  \sum_{(i,j):E_{ij}=1}|A(G)_{ij}-B_{ij}|.
\]
Because $G$ is simple, we obtain $\sum_{i,j}A(G)_{ij}=2e(G)$.
Also, by the counting identity proved above,
\[
  \sum_{i,j}B_{ij}
  =
  \sum_{w=1}^n d_H(w)d_K(w).
\]

It remains to bound the contribution of each error entry.  For every ordered pair $(i,j)$,
\[
  B_{ij}
  =
  |N_H(i)\cap N_K(j)|
  \le
  \min\{d_H(i),d_K(j)\}
  \le
  \Delta_*.
\]
Also $A(G)_{ij}\le 1$.  
Hence we have $|A(G)_{ij}-B_{ij}|\le 1+\Delta_*$
for every error entry $(i,j)$.  
Thus we obtain
\[
  \left|
  2e(G)-\sum_{w=1}^n d_H(w)d_K(w)
  \right|
  \le
  (1+\Delta_*)\|E\|_0.
\]
Since $\|E\|_0\le \eps n^2$, we obtain
\[
  \left|
  2e(G)-\sum_{w=1}^n d_H(w)d_K(w)
  \right|
  \le
  (1+\Delta_*)\eps n^2.
\]
If $\eps=0$, then $\|E\|_0=0$, so the displayed inequality gives $2e(G)=\sum_{w=1}^n d_H(w)d_K(w)$.
\end{proof}

\begin{lem}
\label{lem:row-support}
For simple graphs $H,K$ on $[n]$ and every vertex $i$,
\[
  |\{j:(A(H)A(K))_{ij}>0\}|
  \le
  \sum_{w\in N_H(i)}d_K(w)
  \le
  d_H(i)\Delta(K).
\]
\end{lem}

\begin{proof}
If $(A(H)A(K))_{ij}>0$, then there exists $w\in N_H(i)$ such that
$j\in N_K(w)$.  Thus every such $j$ lies in
\[
  \bigcup_{w\in N_H(i)}N_K(w),
\]
whose size is at most $\sum_{w\in N_H(i)}d_K(w)$ and hence at most
$d_H(i)\Delta(K)$.
\end{proof}

\begin{thm}
\label{thm:degree-tradeoff}
Let $G$ have $m$ edges with an $\eps$-factorization via
$(H,K,E)$, then
\[
  \eps
  \ge
  \max\left\{
  0,\,
  \frac{2m-n\Delta(H)\Delta(K)}{n^2}
  \right\}.
\]
\end{thm}

\begin{proof}
By \Cref{lem:row-support}, each row of $A(H)A(K)$ has at most
$\Delta(H)\Delta(K)$ positive entries.  Hence the number of ordered pairs
$(i,j)$ with $(A(H)A(K))_{ij}>0$ is at most
$n\Delta(H)\Delta(K)$.  The matrix $A(G)$ has exactly $2m$ off-diagonal ones.
Outside the error set, each one of these must occur at a positive entry of $A(H)A(K)$.  
Therefore we obtain  $2m-\|E\|_0\le n\Delta(H)\Delta(K)$.
Thus we have  $\|E\|_0\ge 2m-n\Delta(H)\Delta(K)$.
If the right-hand side is positive, then
$\eps n^2\ge \|E\|_0$ gives the displayed bound.  If it is non-positive, the bound reduces to the trivial inequality $\eps\ge 0$.
\end{proof}

\section{Complete Graphs}
\label{sec:complete}

Complete graphs are a basic benchmark for the theory.  In the exact case, they have already been classified by \citet{maghsoudi2023matrix}. 
We use that classification as a zero-error base case and ask how close the remaining complete graphs are to being factorizable.

\begin{lem}{\rm\cite[Theorem 1]{maghsoudi2023matrix}}
\label{thm:mms-complete}
The complete graph $K_n$ admits an exact matrix product factorization if and only if $n\equiv 1\pmod 4$.
\end{lem}

Here we have an example.
Let $J_m$ denote the all-ones matrix of order $m$, and let $R_m$ denote the backward identity matrix of order $m$. Define
\[
A(H)=J_2\otimes R_{2n},\qquad A(K)=R_2\otimes J_{2n}.
\]
Then, by the mixed-product property of the Kronecker product,
\[
A(H)A(K)
=(J_2R_2)\otimes (R_{2n}J_{2n})
=J_2\otimes J_{2n}
=J_{4n}.
\]

\begin{defn}
Let $\Gamma$ be a finite group with identity element $e$, and let
$S\subseteq \Gamma\sm  \{e\}$ with $S=S^{-1}$.  The Cayley graph
$\Cay(\Gamma,S)$ is the graph with vertex set $\Gamma$ and edge set
\[
  E(\Cay(\Gamma,S))
  =
  \{\{x,y\}:x,y\in\Gamma,\ x\neq y,\ x^{-1}y\in S\}.
\]
When $\Gamma=\mathbb Z_N$ is written additively, this condition becomes
$y-x\in S$.  Thus $S=-S$ makes the graph undirected and $0\notin S$ makes it
loopless.
\end{defn}

For the approximate result we use the explicit cyclic construction behind the
known exact theorem.  When $N=4t+1$, work on $\mathbb Z_N$ and put
\[
  S=\{-1,1\},
  \qquad
  T=\{4q+2,4q+3:q=0,1,\dots,t-1\}.
\]
The Cayley graphs $\Cay(\mathbb Z_N,S)$ and $\Cay(\mathbb Z_N,T)$ give an exact factorization of $K_N$.  
Moreover, each witness contributes exactly $|S||T|=N-1$ ordered pairs.

\begin{thm}
\label{cor:complete-improved-eps}
For every $n\ge 2$, the complete graph $K_n$ admits an
$\eps$-factorization with
\[
\eps \le
\begin{cases}
0, & n\equiv 1 \pmod 4,\\[2mm]
\dfrac{1}{n}, & n\equiv 0 \pmod 4,\\[3mm]
\dfrac{1}{n}, & n\equiv 2 \pmod 4,\\[3mm]
\dfrac{n+1}{n^2}, & n\equiv 3 \pmod 4.
\end{cases}
\]
\end{thm}

\begin{proof}
The case $n\equiv 1\pmod 4$ is exactly
\Cref{thm:mms-complete}.  Hence we assume $n\not\equiv 1\pmod 4$.

We use the standard cyclic exact factorization of complete graphs of order $N\equiv 1\pmod 4$.  
We write $N=4t+1$ and work on the vertex set $\mathbb Z_N$.  
Let $P_N:=\Cay(\mathbb Z_N,\{-1,1\})$
and let $Q_N:=\Cay(\mathbb Z_N,T_N)$, and $T_N:=\{4q+2,4q+3:q=0,1,\dots,t-1\}$.
Then we have $A(P_N)A(Q_N)=A(K_N)$.
In addition, we have 
\[
  d_{P_N}(w)=2
  \qquad\text{and}\qquad
  d_{Q_N}(w)=\frac{N-1}{2}
\]
for every $w\in\mathbb Z_N$.

Now let $D\subseteq \mathbb Z_N$, and put $W:=\mathbb Z_N\sm   D$.
We delete the vertices of $D$ from both factor graphs $P_N$ and $Q_N$, and denote the resulting induced subgraphs on $W$ by $P_N[W]$ and $Q_N[W]$.
Since $A(P_N)A(Q_N)=A(K_N)$, every ordered pair $(u,v)$ with $u,v\in W$ and $u\neq v$ has a unique witness in the original exact factorization.  
After deleting $D$, the only possible errors are the ordered pairs whose unique witness belonged to $D$.  
Thus the number of errors is at most
\[
  \sum_{w\in D}
  |N_{P_N}(w)\sm   D|\,
  |N_{Q_N}(w)\sm   D|.
\]
We now choose $D$ according to the congruence class of $n$.

\medskip
\noindent
\textbf{Case 1: $n\equiv 0\pmod 4$.}
Let $N:=n+1$.
Then $N\equiv 1\pmod 4$ and we set $D:=\{0\}$.
For the deleted witness $0$, we have $|N_{P_N}(0)\sm   D|=2$ and $|N_{Q_N}(0)\sm   D|=\frac{N-1}{2}$.
Hence the number of errors is at most
\[
  2\cdot \frac{N-1}{2}
  =
  N-1
  =
  n.
\]
Therefore we obtain $\eps\le \frac{n}{n^2}=\frac1n$.

\medskip
\noindent
\textbf{Case 2: $n\equiv 3\pmod 4$.}
Let $N:=n+2$.
Then $N\equiv 1\pmod 4$.  
Take two consecutive vertices in the cycle $P_N$:
\[
  D:=\{0,1\}.
\]
For $w=0$, one of the two $P_N$-neighbours of $0$ is deleted, namely $1$.
Thus we have $|N_{P_N}(0)\sm   D|=1$.
Also $1\notin N_{Q_N}(0)$, since the connection set $T_N$ contains neither
$1$ nor $-1$.  Hence
\[
  |N_{Q_N}(0)\sm   D|=\frac{N-1}{2}.
\]
So the deleted witness $0$ contributes at most $\frac{N-1}{2}$ errors.  
The same argument applies to the deleted witness $1$.  Therefore the total number of errors is at most
\[
  2\cdot \frac{N-1}{2}
  =
  N-1
  =
  n+1.
\]
Hence we have $\eps\le \frac{n+1}{n^2}$.

\medskip
\noindent
\textbf{Case 3: $n\equiv 2\pmod 4$.}
Let $N:=n+3$.
Then $N\equiv 1\pmod 4$.  
Take three consecutive vertices in the cycle $P_N$:
\[
  D:=\{-1,0,1\}\subseteq \mathbb Z_N.
\]
The deleted witness $0$ contributes no errors, because both of its
$P_N$-neighbours are deleted:
\[
  N_{P_N}(0)=\{-1,1\}\subseteq D.
\]
Now consider $w=-1$.  
Its $P_N$-neighbours are $-2$ and $0$, and only $0$ is deleted.  
Therefore we have  $|N_{P_N}(-1)\sm   D|=1$.
Also, the vertex $1$ is a $Q_N$-neighbour of $-1$, because
$1-(-1)=2\in T_N$.
Thus at least one $Q_N$-neighbour of $-1$ is deleted, and so
\[
  |N_{Q_N}(-1)\sm   D|
  \le
  \frac{N-1}{2}-1
  =
  \frac{N-3}{2}.
\]
Hence the deleted witness $-1$ contributes at most
$\frac{N-3}{2}$ errors.  
By symmetry, the deleted witness $1$ contributes at most the same number of errors.  Therefore the total number of errors is at most
\[
  2\cdot \frac{N-3}{2}
  =
  N-3
  =
  n.
\]
Hence we obtain $\eps\le \frac{n}{n^2}=\frac1n$.
Combining the three cases with the exact case $n\equiv 1\pmod 4$ proves the
theorem.
\end{proof}

The point of \Cref{cor:complete-improved-eps} is that the exact congruence
obstruction is a zero-error obstruction only.  Every complete graph is within
$O(1/n)$, in matrix Hamming distance, of a product of two adjacency matrices.
No optimality claim is made for the displayed upper bounds.

\section{Blow-ups and Stability }
\label{sec:blowups}

The complete bipartite construction is a special case of a general template
principle.  If a graph is obtained by replacing each vertex of a fixed graph by
an independent set and each edge by a complete bipartite graph, then matching
inside the parts copies the appropriate columns of the adjacency matrix.  Odd
parts create exactly the unmatched-column errors.

\begin{defn}
Let $F$ be a simple graph on vertex set $[r]$.  A graph $G$ with vertex partition  $V(G)=V_1\sqcup V_2\sqcup\cdots\sqcup V_r$
is a \defin{blow-up} of $F$ if:
\begin{enumerate}
  \item $G[V_a]$ is edgeless for each $a\in[r]$;
  \item for $a\neq b$, the bipartite graph between $V_a$ and $V_b$ is complete
  if $ab\in E(F)$, and empty if $ab\notin E(F)$.
\end{enumerate}
\end{defn}

\begin{thm}
\label{thm:blowup-factor}
Let $F$ be a simple graph on $[r]$, and let $G$ be a blow-up of $F$ with parts $V_1,\dots,V_r$ and $n=|V(G)|$.  
Then $G$ admits an $\eps$-factorization with
\[
  \eps
  \le
  \frac{1}{n^2}
  \sum_{\substack{b\in[r]\\ |V_b|\text{ odd}}}
  \sum_{a\in N_F(b)}|V_a|
  \le
  \frac{r}{n}.
\]
Moreover, if every non-isolated part $V_b$ has even size, then $G$ admits a $0$-factorization.
\end{thm}

\begin{proof}
Let $H:=G$.  For each part $V_b$, choose a matching $M_b$ on the vertex set
$V_b$ covering all but at most one vertex.  Let $K$ be the graph whose edge set
is
\[
  E(K):=\bigcup_{b=1}^r M_b.
\]
Thus $K$ has edges only inside the parts.
We fix $i\in V_a$ and $j\in V_b$.  
If $j$ is matched in $M_b$, let $j'$ be its match.  
Then the only possible $K$-neighbour of $j$ is $j'$, so
$(A(H)A(K))_{ij}=A(H)_{ij'}$.
Since $j'\in V_b$ and $G$ is a blow-up of $F$, this equals $A(G)_{ij}$.
If $j$ is unmatched, then column $j$ of $A(K)$ is zero, so
$(A(H)A(K))_{ij}=0$ for every $i$.  
Thus mismatches in column $j$ occur exactly for those
$i\in V_a$ with $ab\in E(F)$.  
The number of such mismatches is $\sum_{a\in N_F(b)}|V_a|$.
Summing over odd parts gives the claimed bound.  Since each inner sum is at most $n$, the coarser estimate $\eps\le r/n$ follows.  
If every non-isolated part has even size, then every part that can create mismatches has a perfect matching. 
In fact, unmatched vertices can occur only in isolated parts and create no errors.
\end{proof}

\begin{cor}
\label{cor:complete-multipartite}
Let $G$ be the complete $r$-partite graph with parts $V_1,\dots,V_r$ and total
order $n$.  Then $A(G)$ admits an $\eps$-factorization with
\[
  \|E\|_0
  \le
  \sum_{\substack{b\in[r]\\ |V_b|\text{ odd}}}(n-|V_b|).
\]
In particular, $\eps\le r/n$.  If every part has even size, then $A(G)$ admits
a $0$-factorization.
\end{cor}

\begin{proof}
A complete $r$-partite graph is a blow-up of $K_r$.  
We invoke \Cref{thm:blowup-factor}.
\end{proof}

\begin{rem}
The exact even-part complete multipartite case is consistent with the complete
multipartite constructions in the exact theory \cite{maghsoudi2023matrix}.  The
point of \Cref{thm:blowup-factor} is that the same matching-inside-parts
argument gives an immediate error count for arbitrary part parities and for
blow-ups of any fixed template.
\end{rem}

\section{Trees and Sparse Bipartite Graphs}
\label{sec:trees}

The exact theory behaves rigidly on trees: it is known that every tree of order at least two is not exactly factorizable \cite{miraftab2025factorability}.
From the approximate point of view, however, the situation is different.  
Since a tree has $n-1$ undirected edges, marking all directed edge entries already gives $2(n-1)$ errors, hence error parameter $O(1/n)$.  
The next result gives a more structured construction: it factors one orientation of almost all edges exactly.

We first provide the construction for bipartite graphs.

\begin{lem}
\label{lem:bipartite-one-sided}
Let $G$ be a bipartite graph with bipartition $X\cup Y$.  Suppose $U\subseteq Y$ has even cardinality.  
Then there exist simple graphs $H,K$ on $V(G)$ such that, for every ordered pair $(a,b)\in V(G)\times V(G)$,
\[
  (A(H)A(K))_{ab}
  =
  \begin{cases}
  1, & \text{if } a\in X,\ b\in U,\ \text{and } ab\in E(G),\\
  0, & \text{otherwise.}
  \end{cases}
\]
\end{lem}

\begin{proof}
Choose a perfect matching $M$ on the vertex set $U$.  
This matching is not required to be a subgraph of $G$.
In fact, it is simply a pairing of the vertices of $U$.  
For $y\in U$, let $\mu(y)$ denote the mate of $y$ in $M$.
We define $K$ by $E(K):=M$.
Thus $K$ has edges only inside $U$.
Now we define $H$ as follows.  
For every $y\in U$ and every $x\in N_G(y)\subseteq X$, put the edge $x\mu(y)$
in $H$.  
No other edges are added to $H$.
We compute the product.  
If $b\in X$, then column $b$ of $A(K)$ is zero,
because $K$ has no edges incident with $X$.  Hence $(A(H)A(K))_{ab}=0$ for every $a\in V(G)$ whenever $b\in X$.
Now let $b=y\in U$.  
The only $K$-neighbour of $y$ is $\mu(y)$.  
Therefore we have
\[
  (A(H)A(K))_{ay}
  =
  A(H)_{a,\mu(y)}.
\]
If $a=x\in X$, then by the definition of $H$ this equals $1$ precisely when $x\in N_G(y)$, or equivalently, precisely when $xy\in E(G)$.
If $a\in Y$, then $A(H)_{a,\mu(y)}=0$,
because $H$ has no edges inside $Y$.

Finally, if $b\in Y\sm   U$, then column $b$ of $A(K)$ is zero, so again $(A(H)A(K))_{ab}=0$ for every $a\in V(G)$.
Therefore $A(H)A(K)$ has value $1$ exactly on the ordered pairs
\[
  (x,y)\in X\times U
  \quad\text{with}\quad
  xy\in E(G),
\]
and is zero everywhere else.  This proves the claim.
\end{proof}

\begin{cor}
\label{cor:bipartite-one-sided-errors}
Let $G$ be a bipartite graph with bipartition $X\cup Y$, with $m=|E(G)|, n=|V(G)|$.
The following hold.
\begin{enumerate}
  \item If $|Y|$ is even, then $A(G)$ admits an $\eps$-factorization with $\|E\|_0=m,\eps=\frac{m}{n^2}$.

  \item If $y_0\in Y$ and $|Y\sm  \{y_0\}|$ is even, then $A(G)$ admits an
  $\eps$-factorization with
  \[
    \|E\|_0=m+d_G(y_0),
    \qquad
    \eps=\frac{m+d_G(y_0)}{n^2}.
  \]
\end{enumerate}
\end{cor}

\begin{proof}
First suppose $|Y|$ is even.  
We invoke \Cref{lem:bipartite-one-sided} with $U=Y$.
The product $A(H)A(K)$ agrees with $A(G)$ on all ordered entries from $X$ to $Y$ and is zero everywhere else.  
Therefore the only errors are the opposite
ordered edge entries from $Y$ to $X$.  
There is one such ordered entry for each edge of $G$, so the number of errors is $\|E\|_0=m$.
Hence we have  $\eps=\frac{m}{n^2}$.

Now suppose $y_0\in Y$ and $|Y\sm \{y_0\}|$ is even.  
Next we apply \Cref{lem:bipartite-one-sided} with $U=Y\sm  \{y_0\}$.
Again, the product misses all opposite ordered edge entries from $Y$ to $X$,
which contributes $m$ errors.  
In addition, it misses the ordered entries $(x,y_0)$ for $x\in N_G(y_0)$.
Because the column corresponding to $y_0$ in $A(K)$ is zero.  
These contribute $d_G(y_0)$ further errors.  
Thus $\|E\|_0=m+d_G(y_0)$,
and therefore
\[
  \eps=\frac{m+d_G(y_0)}{n^2}.\qedhere
\]
\end{proof}

\begin{thm}
\label{thm:trees-eps}
Let $T$ be a tree on $n\ge 2$ vertices.  
Then $T$ admits an $\eps$-factorization with
$\eps\le \frac{1}{n}$.
More precisely, if one bipartition class of $T$ has even size, then $\eps\le \frac{n-1}{n^2}$.
If both bipartition classes have odd size, then
\[
  \eps\le \frac{n}{n^2}=\frac{1}{n}.
\]
\end{thm}

\begin{proof}
Let $X\cup Y$ be the bipartition of $T$.  Since $T$ is a tree, $|E(T)|=n-1$.
First suppose one bipartition class has even size.  After possibly interchanging $X$ and $Y$, assume $|Y|$ is even.  
We invoke \Cref{lem:bipartite-one-sided} with $U=Y$.
The only errors are the reverse directed edges from $Y$ to $X$, and there are $n-1$ of them.  
Therefore we obtain  $\|E\|_0=n-1$,
and hence $\eps\le \frac{n-1}{n^2}$.

Now suppose both bipartition classes have odd size.  Choose a leaf $y_0$ of
$T$; after possibly interchanging the two colour classes, assume $y_0\in Y$.
Then $|Y\sm  \{y_0\}|$ is even.  
We invoke \Cref{lem:bipartite-one-sided} with $U=Y\sm  \{y_0\}$.
The number of errors is $|E(T)|+d_T(y_0)$.
Since $T$ is a tree, $|E(T)|=n-1$, and since $y_0$ is a leaf, $d_T(y_0)=1$.
Therefore we have $\|E\|_0\le (n-1)+1=n$,
so $\eps\le \frac{n}{n^2}=\frac{1}{n}$.
\end{proof}

We close the paper with the following question:
Is there any lower bound on $\varepsilon$ for non-factrizable graphs?

\section*{Acknowledgment}

This research was supported by NSERC.

\section*{Data Availability}

No datasets were generated or analysed during the current study.

\section*{Declarations}

The author declares no conflict of interest.

\bibliographystyle{plainurlnat}
\bibliography{MPF}

\end{document}